\documentclass[a4paper]{amsart}
\usepackage{graphicx}
\usepackage{amssymb}
\usepackage{amsmath}
\usepackage{amsthm}
\usepackage{amscd}
\usepackage[all,2cell]{xy}

\UseAllTwocells \SilentMatrices
\newtheorem{thm}{Theorem}[section]

\newtheorem{cor}[thm]{Corollary}
\newtheorem{lem}[thm]{Lemma}

\newtheorem{prop}[thm]{Proposition}
\theoremstyle{definition}

\theoremstyle{remark}

\numberwithin{equation}{section}

\begin{document}
\title[Retractions and Gorenstein homological properties]{Retractions and Gorenstein homological properties}

\author[Xiao-Wu Chen,  Yu Ye] {Xiao-Wu Chen, Yu Ye}

\thanks{The authors are supported by National Natural Science Foundation of China (No.10971206).}
\subjclass[2010]{16G20, 16E50, 16D90}
\date{\today}

\thanks{E-mail: xwchen$\symbol{64}$mail.ustc.edu.cn, yeyu$\symbol{64}$ustc.edu.cn}
\keywords{retractions, Gorenstein algebras, Gorenstein projective modules, Nakayama algebras, CM-free algebras}%

\maketitle

\dedicatory{}%
\commby{}%

\begin{abstract}
We associate to a localizable module a left retraction of algebras; it is a homological ring epimorphism that preserves singularity categories.  We study the behavior of left retractions with respect to Gorenstein homological properties (for example, being Gorenstein algebras or CM-free).

We apply the results to Nakayama algebras. It turns out that for a connected Nakayama algebra $A$, there exists a connected self-injective Nakayama algebra $A'$ such that there is a sequence of left retractions linking $A$ to $A'$; in particular, the singularity category of $A$ is triangle equivalent to the stable category of $A'$.  We classify connected Nakayama algebras with at most three simple modules according to Gorenstein homological properties.
\end{abstract}

\section{Introduction}
Let $A$ be an artin algebra. The modules we consider here are finitely generated left modules. The study of Gorenstein projective modules, that extend projective modules,  goes back to Auslander and Bridger \cite{ABr}, and it relates to the singularity category of algebras via the work of Buchweitz \cite{Buc}; also see \cite{KV, Hap91, Bel2000, Or04}. Gorenstein projective modules are also known as modules of G-dimension zero \cite{ABr}, totally reflexive modules \cite{AM} or (maximal) Cohen-Macaulay modules \cite{Buc, Bel2000}.

 Gorenstein projective modules play a central role in Gorenstein homological algebra. For example, resolutions by Gorenstein projective modules give rise to the notion of Gorenstein projective dimension for any modules; see \cite{ABr, EJ}. Recall from \cite{Buc, Hap91} that $A$ is Gorenstein if the regular module $A$ has finite injective dimension on both sides. For example, algebras with finite global dimension and self-injective algebras are Gorenstein. Note that $A$ is Gorenstein if and only if each $A$-module has finite Gorenstein projective dimension; in other words, Gorenstein algebras in Gorenstein homological algebra play a similar role as algebras of finite global dimension in classical homological algebra. For Gorenstein algebras, Gorenstein projective modules behave quite nicely.

 An algebra $A$ is called CM-free \cite{Ch11} provided that any Gorenstein projective module is  projective. This implies that Gorenstein homological algebra for $A$ is quite boring. For example, algebras with finite global dimension are CM-free; indeed, an algebra has finite global dimension if and only if it is Gorenstein and CM-free. We are interested in non-Gorenstein CM-free algebras, or equivalently, CM-free algebras with infinite global dimension.

We have the following obvious trichotomy from the point view of Gorenstein homological algebra: any algebra $A$ is either Gorenstein, or non-Gorenstein CM-free, or non-Gorenstein but not CM-free. We are interested in classifying algebras according to this trichotomy. Recall from \cite{Ch11} that a connected algebra with radical square zero is either Gorenstein or non-Gorenstein CM-free; compare \cite{RX}.

In this paper, we show that left retractions of algebras are related to these Gorenstein properties; see Proposition \ref{prop:Gproj} and Theorem \ref{thm:1}. Here, we follow \cite{GL1991, CK} to associate a left retraction of algebras  to a localizable module; left retractions are homological ring epimorphisms and preserve singularity categories. We mention that related results are obtained in \cite{Nag} via a completely different method. We apply there results to Nakayama algebras, where similar ideas might trace back to \cite{Zac} and \cite{BFVZ}. It turns out that for a  connected Nakayama algebra $A$, there is a connected self-injective Nakayama  algebra $A'$ such that there is a sequence of left retractions linking $A$ to $A'$; indeed, if $A$ has infinite global dimension, such an algebra $A'$ is unique up to isomorphism; see Theorem \ref{thm:2}. In particular,  the singularity category of $A$ is triangle equivalent to the stable category of $A'$; see Corollary \ref{cor:NakaSing}.  Finally, we classify Nakayama algebras with at most three simple modules according to the above trichotomy. It turns out that there exists a class of such algebras, that are non-Gorenstein but not CM-free; see Proposition \ref{prop:classification}. For the classification, we use the notion of regular element for Nakayama algebras introduced in \cite{Gus}.

Throughout, $A$ is an artin algebra over a commutative artinian ring $R$. We denote by $A\mbox{-mod}$ the category of finitely generated left $A$-modules. We denote by $n(A)$ the number of pairwise non-isomorphic simple $A$-modules. For a module $M$, $l(M)$ denotes its composition length. We refer to \cite{ARS} for artin algebras.

\section{Left retractions of algebras}

In this section, we recall left retractions of algebras and their basic properties. We study the behavior of
left retractions with respect to Gorenstein homological properties, such as being Gorenstein algebras or CM-free. We point out
that left retractions induce triangle equivalences between singularity categories.

\subsection{Left retractions of categories}

Let $A$ be an artin algebra. Recall that a simple $A$-module $S$ is called \emph{localizable} provided
that ${\rm proj.dim}_A\; S\leq 1$ and ${\rm Ext}_A^1(S, S)=0$. Here, for each $A$-module $X$, ${\rm proj.dim}_A\; X$ denotes
its projective dimension.

For a localizable $A$-module $S$, we consider the \emph{perpendicular subcategory}
$$S^\perp=\{X\in A\mbox{-mod}\; |\; {\rm Hom}_A(S, X)=0={\rm Ext}_A^1(S, X)\}.$$
It is an \emph{exact abelian} subcategory of $A\mbox{-mod}$, that is, the category $S^\perp$ is abelian and the inclusion functor $i\colon S^\perp \rightarrow A\mbox{-mod}$ is exact. A nontrivial fact is that the
functor $i$ admits a right adjoint $i_\lambda$ which is exact. For details, see \cite[Proposition 3.2]{GL1991}.

 The inclusion functor $i\colon  S^\perp \rightarrow A\mbox{-mod}$ is called a \emph{left expansion} of categories, while the right adjoint $i_\lambda\colon A\mbox{-mod}\rightarrow S^\perp$ is called a \emph{left retraction}; see \cite{CK}.

Denote by ${\rm add}\; S$ the full subcategory of $A\mbox{-mod}$ consisting of  finite direct sums of the localizable module $S$. It is a \emph{Serre subcategory}, that is, it is closed under extensions, submodules and quotient
modules. Consider the quotient abelian category $A\mbox{-mod}/{{\rm add}\; S}$ in the sense of Gabriel \cite{G}. Then the functor $i_\lambda$ induces an equivalence $A\mbox{-mod}/{{\rm add}\; S}\simeq S^\perp$, in other words, the functor $i_\lambda$ identifies with
the quotient functor $q\colon A\mbox{-mod}\rightarrow A\mbox{-mod}/{{\rm add}\; S}$; see \cite[Lemma 3.1.2]{CK}.

 Set $\Delta={\rm End}_A(S)^{\rm op}$ to be the opposite algebra of the endomorphism algebra of $S$; it is a division algebra.  There is an equivalence ${\rm Hom}_A(S, -)\colon {\rm add}\; S \stackrel{\sim}\longrightarrow \Delta\mbox{-mod}$ of categories.

Let us recall from \cite[Section 3]{GL1991} the construction of the functor $i_\lambda\colon A\mbox{-mod}\rightarrow S^\perp$. For an $A$-module $X$, we take an exact sequence $0\rightarrow X\rightarrow X'\rightarrow S^{\oplus m_1} \rightarrow 0$ such that
${\rm Ext}_A^1(S, X')=0$ and $m_1={\rm dim}_\Delta\;  {\rm Ext}_A^1(S, X)$. Then we take an exact sequence $0\rightarrow S^{\oplus m_2} \rightarrow X'\rightarrow X''\rightarrow 0$ such that  ${\rm Hom}_A(S, S^{\oplus m_2})\rightarrow {\rm Hom}_A(S, X')$ is an isomorphism. It follows that $X''$ lies in $S^\perp$. The composite $X\rightarrow X'\rightarrow X''$ is the universal morphism of $X$ to $S^\perp$, that is, $i_\lambda (X)=X''$. We conclude with an exact sequence
\begin{align}\label{equ:1}
0\longrightarrow t(X) \longrightarrow X\stackrel{\eta_X} \longrightarrow i i_\lambda(X) \longrightarrow t'(X)\longrightarrow 0.
\end{align}
Here, both $t(X)$ and $t'(X)$ lie in ${\rm add}\; S$. Moreover, this yields a functor $t \colon A\mbox{-mod}\rightarrow {\rm add}\; S$, which is right adjoint to the inclusion functor ${\rm inc}\colon {\rm add}\; S \rightarrow A\mbox{-mod}$.

We observe that $\eta$ is the unit of the adjoint pair $(i_\lambda, i)$. Consider the endofunctor $L=ii_\lambda\colon A\mbox{-mod}\rightarrow A\mbox{-mod}$. Since the functor $i$ is fully faithful, the counit $i_\lambda  i\rightarrow {\rm Id}_{S^\perp}$ is an isomorphism. It follows that for each $A$-module $X$, $L(\eta_X)=\eta_{L(X)}\colon L(X)\rightarrow L^2(X)$ is an isomorphism. In other words, the pair $(L, \eta\colon {\rm Id}_{A\mbox{-}{\rm mod}}\rightarrow L)$ is a \emph{localization functor }in the sense of \cite[Section 3]{BIK}.

In summary, for a localizable $A$-module $S$, we have the following diagram of functors

\begin{equation}\label{eq:rightrec}
\xymatrixrowsep{3pc} \xymatrixcolsep{1.5pc}\xymatrix{
  \Delta\mbox{-mod} \simeq {\rm add}\; S \;\ar[rr]^-{{\rm inc}}&&\;  A\mbox{-mod} \; \ar[rr]^-{i_\lambda}
  \ar@<1.2ex>[ll]^-{t}&&
  \;S^\perp. \ar@<1.2ex>[ll]^-{i}
}\end{equation}

\subsection{Left retractions of algebras}

Let $(L, \eta)$ be the localization functor associated to a localizable module $S$. Consider the $A$-module $L(A)$. Recall the isomorphism $L(\eta_A)=\eta_{L(A)}\colon L(A)\rightarrow L^2(A)$.  For each element $x\in L(A)$, denote by $r_x\colon A\rightarrow L(A)$ the morphism given by $r_x(a)=a.x$; here, the dot ``." denotes the $A$-module action. For $x, y\in L(A)$, we define $x\star y=L(\eta_A)^{-1}(L(r_y)(x))\in L(A)$. This gives rise to an algebra structure on $L(A)$. Observe that
$\eta_A(a)\star y=a.y$ for $a\in A$. In particular, $\eta_A\colon A\rightarrow L(A)$ is a homomorphism of algebras and
the left $A$-module structure on $L(A)$ coincides with the one induced from $\eta_A$.

We call the algebra homomorphism $\eta_A\colon A\rightarrow L(A)$  the \emph{left retraction} associated to the localizable module $S$.

We observe that the left retraction $\eta_A$ is surjective if and only if $t'(A)=0$, or equivalently, ${\rm Ext}_A^1(S, A)=0$. This happens if the simple module $S$ is projective, in which case the algebra $A$ is Morita equivalent to  a one-point (co-)extension of $L(A)$; compare \cite[III. 2]{ARS}.

The following result is well known; compare \cite[Corollary 3.9]{GL1991}.

\begin{lem}\label{lem:quotient} Keep the notation as above. Then we have the following statements:
\begin{enumerate}
\item there is an algebra isomorphism $\Phi\colon L(A)\stackrel{\sim}\longrightarrow {\rm End}_{S^\perp}(i_\lambda(A))^{\rm op} $ such that $i(\Phi(x))=L(\eta_A)^{-1}\circ L(r_x)$;
\item there is an equivalence $S^\perp \simeq L(A)\mbox{-{\rm mod}}$ of categories; moreover, the functor $i$ identifies with
${\rm Hom}_{L(A)}(L(A), -)$, and $i_\lambda$ with $L(A)\otimes_A-$.
\end{enumerate}
\end{lem}

The fully-faithfulness of the functor $i\simeq {\rm Hom}_{L(A)}(L(A), -)\colon L(A)\mbox{-mod}\rightarrow A\mbox{-mod}$ implies
that $\eta_A\colon A\rightarrow L(A)$ is a ring epimorphism. The exactness of the functor $i_\lambda\simeq L(A)\otimes_A-$
implies that the right $A$-module $L(A)$ is flat and then projective. In particular,  the left retraction $\eta_A\colon A\rightarrow L(A)$  is a \emph{left localization} in the sense of \cite{Sil}. It follows that $\eta_A$ is a homological ring epimorphism; see \cite[Corollary  4.7]{GL1991}. We observe from the exact sequence (\ref{equ:1}) that ${\rm proj.dim}_A\; L(A)\leq 2$.

We mention that there exists an idempotent $e$ in $A$ such that ${\rm add}\; S$ is the kernel of the \emph{Schur functor} $S_e=eA\otimes_A-\colon A\mbox{-mod}\rightarrow eAe\mbox{-mod}$. Then $S_e$ induces an equivalence $A\mbox{-mod}/{{\rm add}\; S}\simeq eAe\mbox{-mod}$, which is further equivalent to $S^\perp$. Then Lemma \ref{lem:quotient}(2) implies that $L(A)$ and $eAe$ are Morita equivalent.

From the above discussion, we may identify the Schur functor $S_e$ with the quotient functor $q$, and thus with $i_\lambda$. In particular, the right adjoint of $S_e$, that is ${\rm Hom}_{eAe}(eA, -)$, is exact. Hence, the left $eAe$-module $eA$ is projective.

\begin{proof}
For (1), we observe that $L(A)\simeq {\rm Hom}_A(A, ii_\lambda(A))\simeq {\rm Hom}_{S^{\perp}}(i_\lambda(A), i_\lambda(A))$, where the second isomorphism is by the adjoint pair $(i_\lambda, i)$.  The composition is  $\Phi$; one checks directly that it gives a homomorphism of algebras.

For (2), we observe that $i_\lambda(A)$ is projective in $S^\perp$, since $i$ is exact; see \cite[Proposition 2.3.10]{Wei}.
Identifying $i_\lambda$ with the quotient functor $q\colon A\mbox{-mod}\rightarrow A\mbox{-mod}/{{\rm add}\; S}$, we infer that each object
in $S^\perp$ is a quotient of finite direct sums of copies of $i_\lambda(A)$. Hence, $i_\lambda(A)$ is a projective generator of $S^\perp$. Then using (1), we deduce the equivalence of categories.
\end{proof}

\subsection{Simple modules} Let $n(A)=n$ be the number of pairwise non-isomorphic simple $A$-modules, and let $\{S_1, S_2, \cdots, S_{n-1}, S_n=S\}$ be a complete set of representations of pairwise non-isomorphic simple $A$-modules. Set $\Delta_j={\rm End}_A(S_j)^{\rm op}$; they are division algebras.

We assume that $S_n=S$ is localizable. Let $P_j=P(S_j)$ and $I_j=I(S_j)$ be the projective cover  and the injective envelop of $S_j$, respectively. Since $S_n$ is localizable,
its minimal projective presentation takes the following form
\begin{align}\label{equ:S}
0\longrightarrow \bigoplus_{j=1}^{n-1} P_j^{\oplus d_j} \longrightarrow P_n\longrightarrow S_n\longrightarrow 0.
\end{align}
Here, each $d_j\geq 0$.

Recall that the \emph{Cartan matrix} $C_A=(c_{jk})$ of $A$ is an $n\times n$ matrix such that $c_{jk}$
is the multiplicity of $S_j$ in a composition series of $P_k$. Consider the homomorphism $C_A\colon \mathbb{Z}^n\rightarrow \mathbb{Z}^n$ between free abelian groups induced by multiplication of $C_A$ from the left; we view elements in $\mathbb{Z}^n$ as  column vectors. Then ${\rm Cok}\; C_A$ is a finitely generated abelian group such that ${\rm rk} ({\rm Cok}\; C_A)=n-{\rm rank}\; C_A$, where we denote by ``rk" the rank of an abelian group.

\begin{prop}\label{prop:simple} Consider the left retraction $\eta_A\colon A\rightarrow L(A)$ associated to $S_n=S$. Identify $S^\perp$ with
$L(A)\mbox{-{\rm mod}}$. Then we have the following statements:
\begin{enumerate}
\item the set $\{i_\lambda(S_1), i_\lambda(S_2), \cdots, i_\lambda(S_{n-1})\}$ is a complete set of representatives of pairwise non-isomorphic simple $L(A)$-modules;
\item there is an isomorphism $\Delta_j\simeq {\rm End}_{L(A)}(i_\lambda(S_j))^{\rm op}$ induced by $i_\lambda$ for each $1\leq j\leq n-1$;
\item the $L(A)$-module $i_\lambda(P_j)$ is the projective cover of $i_\lambda(S_j)$ for $1\leq j\leq n-1$;
\item the $L(A)$-module $i_\lambda(I_j)$ is the injective envelop of $i_\lambda(S_j)$ for $1\leq j\leq n-1$;
\item the Cartan matrix $C_{L(A)}$ of $L(A)$ is the minor of $C_A$ by deleting the $n$-th row and $n$-th column; consequently, we have that ${\rm det}\; C_{L(A)}={\rm det} \; C_A$, ${\rm rank}\; C_{L(A)}={\rm rank}\; C_{A}-1$ and an isomorphism $ {\rm Cok}\; C_{L(A)}\simeq {\rm Cok}\; C_A$ of abelian groups.
\end{enumerate}
\end{prop}

\begin{proof}
We identify $i_\lambda$ with the quotient functor  $q\colon A\mbox{-mod}\rightarrow A\mbox{-mod}/{{\rm add}\; S}$. Observe that each $L(A)$-module has a submodule series with factors $i_\lambda(S_j)$ for $1\leq j\leq n-1$. Then
(1) and (2) follow from a general fact: given  any abelian category $\mathcal{A}$, a Serre subcategory $\mathcal{C}$ and two
simple objects $S, S'$ of $\mathcal{A}$ that are not in $\mathcal{C}$, the quotient functor $q\colon \mathcal{A}\rightarrow \mathcal{A}/\mathcal{C}$
sends $S$ and $S'$ to simple objects, and induces an isomorphism ${\rm Hom}_\mathcal{A}(S, S')\rightarrow {\rm Hom}_{\mathcal{A}/\mathcal{C}}(q(S),q(S'))$.

For (3), we recall that each $P_j$ has a unique maximal submodule $R_j$ such that $P_j/R_j\simeq S_j$. It follows that $i_\lambda(R_j)$ is the unique maximal submodule of $i_\lambda(P_j)$. Here, we use another general fact: for an object $X$ in $\mathcal{A}$, each subobject
of $q(X)$ is induced by a subobject of $X$.  Then (3) follows from the facts that  $i_\lambda (P_j)$ is projective and that $i_\lambda(P_j)/i_\lambda(R_j)\simeq  i_\lambda(S_j)$.

For (4), we observe that each $I_j$ lies in $S^\perp$, and it is injective and indecomposable in $S^\perp$. We identify $i_\lambda(I_j)$ with $I_j$. Consider the embedding $S_j\rightarrow I_j$. We infer that the induced embedding $i_\lambda(S_j)\rightarrow i_\lambda(I_j)$ is an injective hull.

For (5), we observe that the multiplicity of $i_\lambda(S_j)$ in a composition series of $i_\lambda(P_k)$ is the same of the one of
$S_j$ in a composition series of $P_k$. For the equality of determinants, we observe from (\ref{equ:S}) that the last column of
$C_A$ is the sum of $e_n=(0, 0, \cdots, 0, 1)^t$ with a linear combination of the $j$-th rows for $1\leq j \leq n-1$; here, ``t" denotes the transpose. Then all the statements follow immediately. \end{proof}

Recall the \emph{valued quiver} $Q_A$ of an algebra $A$. Let $\{S_1, S_2, \cdots, S_{n-1}, S_n\}$  and  $\{\Delta_1, \Delta_2, \cdots, \Delta_{n-1}, \Delta_{n}\}$ be as above. Observe that ${\rm Ext}_A^1(S_j, S_k)$ has a natural $\Delta_j$-$\Delta_k$-bimodule structure.  The quiver $Q_A$ has the vertex set $\{S_1, S_2, \cdots, S_{n-1}, S_n\}$, and there is an arrow from $S_j$ to $S_k$ whenever ${\rm Ext}_A^1(S_j, S_k)\neq 0$; this arrow is endowed with a valuation $({\rm dim}_{{\Delta_k}^{\rm op}}\; {\rm Ext}_A^1(S_j, S_k), {\rm dim}_{\Delta_j}\; {\rm Ext}_A^1(S_j, S_k))$. The valuation of $Q_A$ is \emph{trivial} if all the valuations of arrows are $(1, 1)$.

Consider the left localization $\eta_A\colon A\rightarrow L(A)$. By Proposition \ref{prop:simple} the vertex set of
the valued quiver $Q_{L(A)}$ of $L(A)$ is obtained by deleting $S_n$ from the one of $Q_A$. The following result implies the quiver
 obtained from $Q_A$ by deleting $S_n$ is \emph{smaller} than $Q_{L(A)}$ in the sense of \cite[p.244]{ARS}; their difference
only appears in the neighborhood of $S_n$.

 \begin{lem}\label{lem:quiver}
 For $1\leq j, k\leq n-1$, there is an exact sequence
 \begin{align*}
 0\longrightarrow {\rm Ext}_A^1(S_j, S_k)\stackrel{\xi}\longrightarrow {\rm Ext}_{L(A)}^1 (i_\lambda(S_j), i_\lambda(S_k)) \longrightarrow
 {\rm Ext}_A^1(S_j, t'(S_k))
 \end{align*}
 where the module $t'(S_k)$ is defined in (\ref{equ:1}), and $\xi$ is induced by $i_\lambda$.

  In particular, the map $\xi$ is an isomorphism provided that in $Q_A$, there is no arrow from $S_j$ to $S_n$, or no arrow from $S_n$ to $S_k$.
 \end{lem}

\begin{proof}
Recall $S_n=S$ and that in the adjoint pair $(i_\lambda, i)$ both functors are exact. Then we have $ {\rm Ext}_{L(A)}^1 (i_\lambda(S_j), i_\lambda(S_k))\simeq {\rm Ext}_A^1(S_j, ii_\lambda (S_k))$; see \cite[Lemma 2.3.1]{CK}. The exact sequence (\ref{equ:1}) for $S_k$ takes
the form $0\rightarrow S_k\rightarrow ii_\lambda(S_k)\rightarrow t'(S_k)\rightarrow 0$, since $t(S_k)=0$. Applying the
functor ${\rm Hom}_A(S_j, -)$ to this sequence, we are done. Here, one might notice that ${\rm Hom}_A(S_j, t'(S_k))=0$, since
$t'(S_k)\in  {\rm add}\; S$. For the last statement, we note that $t'(S_k)=0$ if and only if ${\rm Ext}_A^1(S, S_k)=0$.
\end{proof}

\subsection{Homological properties} For an artin  algebra $A$, we denote by ${\rm gl.dim}\; A$ its global dimension.
Recall its \emph{finitistic dimension} ${\rm fin.dim}\; A={\rm sup}\{{\rm proj.dim}_A\; X\; |\; X\in A\mbox{-mod} \mbox{ with } {\rm proj.dim}_A\; X<\infty \}$. For an algebra $A$ with finite global dimension, we have ${\rm gl.dim}\; A={\rm fin.dim}\; A$. We have the following result; compare \cite[Lemma 4]{BFVZ}.

\begin{lem}\label{lem:homo}
Let $\eta_A\colon A\rightarrow L(A)$ be the left retraction associated to $S_n=S$. Identify $S^\perp$ with
$L(A)\mbox{-{\rm mod}}$. Let $X\in A\mbox{-{\rm mod}}$. Then the following statements hold:
\begin{enumerate}
\item ${\rm proj.dim}_{L(A)} \; i_\lambda(X) \leq {\rm proj.dim}_A\; X\leq {\rm proj.dim}_{L(A)} \; i_\lambda(X)+2$;
\item ${\rm gl.dim}\; L(A)\leq {\rm gl.dim}\; A\leq {\rm gl.dim}\; L(A)+2$;
\item ${\rm fin.dim}\; L(A)\leq {\rm fin.dim}\; A\leq {\rm fin.dim}\; L(A)+2$.
\end{enumerate}
\end{lem}

\begin{proof}
It suffices to show (1). Since $i_\lambda$ sends projective $A$-modules to projective $L(A)$-modules, the left inequality follows. For the right one, recall that $i(L(A))=L(A)$, viewed as a left $A$-module, satisfies that ${\rm proj.dim}_A\; L(A)\leq 2$. Then for each projective $L(A)$-module $Q$, ${\rm proj.dim}_A\; i(Q)\leq 2$. This implies ${\rm proj.dim}_A\; ii_\lambda (X)\leq {\rm proj.dim}_{L(A)} \; i_\lambda(X)+2$. Using the fact that ${\rm proj.dim}_A\; S\leq 1$,  the result follows from the exact sequence (\ref{equ:1}).
\end{proof}

Let $P^\bullet=\cdots \rightarrow P^{-1}\stackrel{d^{-1}}\rightarrow P^0\stackrel{d^0}\rightarrow P^{1}\stackrel{d^1}\rightarrow  P^2\rightarrow\cdots$ be an unbounded complex of projective $A$-modules. Recall that $Z^1(P^\bullet)={\rm Ker}\; d^1$ is the first cocycle. Following \cite{AM}, the complex $P^\bullet$ is \emph{totally acyclic} provided that it is acyclic and the dual complex $(P^\bullet)^*={\rm Hom}_A(P^\bullet, A)$ is also acyclic. An $A$-module $X$ is \emph{Gorenstein projective} provided that there exists a totally acyclic complex $P^\bullet$ such that $Z^1(P^\bullet)\simeq X$; such a complex $P^\bullet$ is called a \emph{complete resolution} of $X$; see \cite{EJ}.

We denote by $A\mbox{-Gproj}$ the full subcategory consisting of Gorenstein projective $A$-modules. Any projective module $P$ is \emph{Gorenstein projective}, since its complete resolution can be taken as $\cdots \rightarrow 0\rightarrow P\stackrel{{\rm Id}_P}\rightarrow P\rightarrow 0\rightarrow \cdots$. The subcategory $A\mbox{-Gproj}$ is closed under extensions, and thus carries a natural exact structure in the sense of Quillen \cite{Qui73}. As an exact category,  $A\mbox{-Gproj}$ is Frobenius, whose projective-injective objects are precisely projective $A$-modules. By \cite[Theorem I.2.8]{Hap88}, the \emph{stable category  } $A\mbox{-\underline{Gproj}}$ modulo projective modules is a triangulated category: the translation functor is induced by a quasi-inverse of the syzygy functor, and triangles are induced by short exact sequences with terms in $A\mbox{-Gproj}$. If $A$ is self-injective, we have that $A\mbox{-Gproj}=A\mbox{-mod}$, and thus the resulting triangulated category $A\mbox{-\underline{Gproj}}$ coincides with the stable category $A\mbox{-\underline{mod}}$.

The following fact is well known.

\begin{lem} \label{lem:Gproj}
Let $X$ be an indecomposable Gorenstein projective $A$-module which is non-projective. Then there exists an exact sequence
$0\rightarrow X\rightarrow P\rightarrow X'\rightarrow 0$ such that $P$ is projective and $X'$ is indecomposable Gorenstein projective which is non-projective. Moreover, the morphism $P\rightarrow X'$ is a projective cover of $X'$.
\end{lem}

\begin{proof}
From the definition, there is an exact sequence $0\rightarrow X\rightarrow P\rightarrow X'\rightarrow 0$ with $P$ projective and $X'$ Gorenstein projective. We may take the sequence such that $P$ has the minimal length.  This implies that $X'$ has no projective direct summands and then the indecomposable module $X$ is the first syzygy of $X'$. So we have that $X'$ is indecomposable and that $P\rightarrow X'$ is a projective cover.
\end{proof}

Recall from \cite{Ch11} that an artin algebra $A$ is \emph{CM-free} provided that each Gorenstein projective $A$-module is projective.

\begin{prop}\label{prop:Gproj}
Let $\eta_A\colon A\rightarrow L(A)$ be the left retraction associated to $S_n=S$. Identify $S^\perp$ with
$L(A)\mbox{-{\rm mod}}$. Then for any $X\in A\mbox{-{\rm Gproj}}$, we have $i_\lambda(X)\in L(A)\mbox{-{\rm Gproj}}$. In particular,
if $L(A)$ is CM-free, so is $A$.
\end{prop}

\begin{proof}
The last statement follows from the first one: if $i_\lambda(X)$ is projective, the $A$-module $X$ has finite projective dimension; see Lemma \ref{lem:homo}(1);  then it suffices to recall from  \cite[Proposition 10.2.3]{EJ} that a Gorenstein projective module of finite projective dimension is necessarily projective.

For the first statement, it suffices to show that for a totally acyclic complex $P^\bullet$ of $A$-modules, the complex $i_\lambda (P^\bullet)$ is totally acyclic. Observe that the complex $i_\lambda (P^\bullet)$ consists of projective $L(A)$-modules and is acyclic. From the adjoint pair $(i_\lambda, i)$, we have the isomorphism ${\rm Hom}_{L(A)}(i_\lambda(P^\bullet), L(A))\simeq {\rm Hom}_A(P^\bullet, L(A))$ of complexes.
Since the left $A$-module $L(A)$ has projective dimension at most two and $P^\bullet$ is totally acyclic, by \cite[Lemma 2.4(iii)]{AM}
the complex ${\rm Hom}_A(P^\bullet, L(A))$ is acyclic. It follows that the complex $i_\lambda (P^\bullet)$ of $L(A)$-modules is totally acyclic.
\end{proof}

\subsection{Gorenstein algebras} Recall from \cite{Buc, Hap91} that an artin algebra $A$ is \emph{Gorenstein} provided that the regular $A$-module
has finite injective dimension on both sides. In this case, both injective dimensions are the same, which are called the \emph{virtual dimension }of $A$ and denoted by ${\rm v.dim}\; A$. Observe that for a Gorenstein algebra $A$, a module has finite injective dimension if and only if it has finite projective dimension; moreover, we have ${\rm v.dim}\; A={\rm fin.dim}\; A$.

An artin algebra $A$ is selfinjective, if and only if it is Gorenstein with virtual dimension zero, if and only if $A\mbox{-Gproj}=A\mbox{-mod}$. On the other hand, $A$ has finite global dimension if and only if it is Gorenstein and CM-free; compare \cite[Theorem 6.9($\iota$)]{Bel2000}.

The following result is well known; see \cite[Theorem 6.9(13)]{Bel2000} and compare \cite[Theorem 3.2]{AM}.

\begin{lem}\label{lem:well}
Let $A$ be an artin algebra and let $d\geq 0$. Then $A$ is Gorenstein with ${\rm v.dim}\; A\leq d$ if and only if each $A$-module $X$ fits into an exact sequence $0\rightarrow G\rightarrow P^{-(d-1)}\rightarrow \cdots \rightarrow P^{-1}\rightarrow P^0\rightarrow X\rightarrow 0$ with each $P^{-j}$ projective and $G\in A\mbox{-{\rm Gproj}}$. \hfill $\square$
\end{lem}

The next result relates the Gorensteinness of $A$ and $L(A)$ in a left retraction.  We mention that this result is related  to \cite[Proposition 6(3)]{Nag}; consult the remarks after Lemma \ref{lem:quotient}.

\begin{thm}\label{thm:1}
Let $\eta_A\colon A\rightarrow L(A)$ be the left retraction associated to the localizable module $S_n=S$. Identify $S^\perp$ with
$L(A)\mbox{-{\rm mod}}$.  Denote by $I_n$ the injective envelop  of $S_n$.  Then $A$ is Gorenstein if and only if $L(A)$ is Gorenstein and ${\rm proj.dim}_{L(A)} \; i_\lambda(I_n)<\infty$.
\end{thm}

\begin{proof}
For the ``only if" part, let $A$ be Gorenstein with ${\rm v.dim}\; A=d$. Recall that each $L(A)$-module $Y$ is of the form $i_\lambda(X)$ for an $A$-module $X$. Take an exact sequence $0\rightarrow G\rightarrow P^{-(d-1)}\rightarrow \cdots \rightarrow P^{-1}\rightarrow P^0\rightarrow X\rightarrow 0$ of $A$-modules with each $P^{-j}$ projective and $G\in A\mbox{-Gproj}$. We apply  the functor $i_\lambda$ to this sequence.  By Proposition \ref{prop:Gproj} and Lemma \ref{lem:well}, we get an exact sequence for $Y$, which implies that $L(A)$ is Gorenstein with ${\rm v.dim}\; L(A)\leq d$. Observe that $I_n$ has finite projective dimension, and then by Lemma \ref{lem:homo}(1) we have ${\rm proj.dim}_{L(A)} \; i_\lambda(I_n)<\infty$.

For the ``if" part, assume that  $L(A)$ is Gorenstein and ${\rm proj.dim}_{L(A)} \; i_\lambda(I_n)<\infty$. For the Gorensteinness of $A$, it suffices to show that the regular $A$-module $_AA$ has finite injective dimension and each injective $A$-module $I_j$ has finite projective dimension.

We observe that $i_\lambda(I_n)$ has finite injective dimension and so does the $A$-module $ii_\lambda (I_n)$. Here, we recall that $i$ sends injective $L(A)$-modules to injective $A$-modules. We observe that the exact sequence (\ref{equ:1}) for $I_n$ has the form
$$0\longrightarrow S_n\longrightarrow I_n\longrightarrow  ii_\lambda(I_n)\longrightarrow 0.$$
Here, we use that $t(I_n)\simeq S_n$ and $t'(I_n)=0$. It follows that $S_n$ has finite injective dimension. Consider
the regular $A$-module $_AA$. Since $i_\lambda(A)=L(A)$ has finite injective dimension, the $A$-module $ii_\lambda(A)$ has finite
injective dimension. Then the exact sequence (\ref{equ:1}) for $_AA$ implies that $_AA$ has finite injective dimension.

By Lemma \ref{lem:homo}(1) the $A$-module $I_n$ has finite projective dimension. It remains to prove that for each $1\leq j\leq n-1$,  the $A$-module $I_j$ has finite projective dimension. Proposition \ref{prop:simple}(4) implies that the $L(A)$-module $i_\lambda(I_j)$ is injective.  Since $L(A)$ is Gorenstein,  $i_\lambda(I_j)$ has finite projective dimension. Then applying Lemma \ref{lem:homo}(1), we are done.
\end{proof}

We observe the following consequence.

\begin{cor}
Let $S$ and $S'$ be two non-isomorphic localizable $A$-modules. Denote by $\eta_A\colon A\rightarrow L(A)$ and
$\eta'_A\colon A\rightarrow L'(A)$ the corresponding left retractions. Then $A$ is Gorenstein if and only if both $L(A)$ and
$L'(A)$ are Gorenstein.
\end{cor}

\begin{proof}
The ``only if" part follows from Theorem \ref{thm:1}.

For the ``if" part, assume that $S_n=S$ and $S_1=S'$. Consider the injective
envelop $I_n$ for $S$. Applying Proposition \ref{prop:simple}(4) to $S'$, we have that the $L'(A)$-module $i_\lambda'(I_n)$ is injective and then has finite projective dimension. By Lemma \ref{lem:homo}(1) for $S'$, the $A$-module $I_n$ has finite projective dimension. It follows from  Lemma \ref{lem:homo}(1) for $S$ that $i_\lambda(I_n)$ has finite projective dimension. Applying Theorem \ref{thm:1} for $S_n=S$, we are done.
\end{proof}

\subsection{Recollements and singular equivalences}

We will show that a localizable module induces a recollement \cite{BBD} of derived categories and a triangle equivalence between singularity categories.

For an artin algebra $A$, denote by
$\mathbf{D}^b(A\mbox{-mod})$ the bounded derived category. Recall that $A\mbox{-mod}$ embeds into $\mathbf{D}^b(A\mbox{-mod})$ by identifying an $A$-module with the stalk complex concentrated on degree zero.  For an exact functor $F\colon A\mbox{-mod}\rightarrow A'\mbox{-mod}$ between module categories, we denote by $F^*\colon \mathbf{D}^b(A\mbox{-mod})\rightarrow \mathbf{D}^b(A'\mbox{-mod})$ its natural extension on complexes.

Let $S$ be a localizable $A$-module. Consider the left localization $\eta_A\colon A\rightarrow L(A)$. By identifying $S^\perp$ with  $L(A)\mbox{-{\rm mod}}$, the functor $i\colon A\mbox{-mod}\rightarrow S^\perp$ identifies with ${\rm Hom}_{L(A)}(L(A), -)$, which is further isomorphic to $L(A)\otimes_{L(A)}-$. It follows that $i$ admits a right adjoint $i_\rho={\rm Hom}_A(L(A), -)$. Recall the equivalence ${\rm add}\; S\simeq \Delta\mbox{-mod}$ with $\Delta={\rm End}_A(S)^{\rm op}$.

\begin{lem}\label{lem:recolle}
Keep the notation as above. Then we have a recollement of derived categories
\[\xymatrixrowsep{3pc} \xymatrixcolsep{2pc}\xymatrix{
\mathbf{D}^b(L(A)\mbox{-{\rm mod}}) \;\ar[rr]|-{i^*} &&\;\mathbf{D}^b(A\mbox{-{\rm mod}}) \; \ar[rr]|-{\mathbf{R}^b t}
\ar@<1.2ex>[ll]^-{\mathbf{R}^b i_\rho}\ar@<-1.2ex>[ll]_-{i^*_\lambda}&&
\; \mathbf{D}^b(\Delta\mbox{-{\rm mod}}) \ar@<1.2ex>[ll]^-{}\ar@<-1.2ex>[ll]_-{{\rm inc}^*}
}\]
\end{lem}

We mention that both functors $i_\rho\colon A\mbox{-mod}\rightarrow L(A)\mbox{-mod}$ and $t\colon A\mbox{-mod}\rightarrow {\rm add}\; S \simeq \Delta\mbox{-mod}$ are left exact. The notation $\mathbf{R}^b$ means the right (bounded) derived functor.

\begin{proof}
The proof is similar to \cite[Proposition 3.3.2]{CK}. Here, it suffices to note that $\mathbf{R}^b \rho$ is well defined, since the left $A$-module $L(A)$ has finite projective dimension; moreover, $\mathbf{R}^b \rho$ is left adjoint to $i^*$.  Similar remarks hold for $t$, since $t\simeq {\rm Hom}_A(S, -)$.
\end{proof}

For an artin algebra $A$, the \emph{singularity category} $\mathbf{D}_{\rm sg}(A)$ is the Verdier quotient category of
$\mathbf{D}^b(A\mbox{-mod})$ by the triangulated subcategory ${\rm perf}(A)$ formed by perfect complexes; see \cite{Buc, Or04}. Here, a bounded complex of $A$-modules is\emph{ perfect} provided that it is isomorphic to a bounded complex of projective $A$-modules in  $\mathbf{D}^b(A\mbox{-mod})$. The triangulated subcategory ${\rm per}(A)$ is \emph{thick}, that is, it is closed under taking direct summands.

Consider the following composite of functors
$$G_A\colon A\mbox{-Gproj}\; \hookrightarrow A\mbox{-mod} \longrightarrow \mathbf{D}^b(A\mbox{-mod})
\longrightarrow \mathbf{D}_{\rm sg}(A)$$ where from the left side,
the first functor is the inclusion, the second identifies modules
with stalk complexes concentrated on degree zero,  and the last is the quotient functor. Observe
that the additive functor $G_A$ vanishes on projective modules and
then induces uniquely an additive  functor
$A\mbox{-\underline{Gproj}} \rightarrow \mathbf{D}_{\rm sg}(A)$,
which is still denoted by $G_A$.

\vskip 5pt

We recall the following fundamental result.

\begin{lem} \label{lem:BuchweitzHappel}  The functor $G_A\colon
A\mbox{-\underline{\rm Gproj}}\rightarrow \mathbf{D}_{\rm sg}(A)$
is a fully faithful triangle functor. Moreover, the algebra $A$ is
Gorenstein if and only if $G_A$ is dense and thus a triangle
equivalence. In particular, if $A$ is self-injective, we have a triangle equivalence $\mathbf{D}_{\rm sg}(A)\simeq A\mbox{-\underline{\rm mod}}$.
\end{lem}

\begin{proof}
The result is due to Buchweitz \cite[Theorem 4.4.1]{Buc} and independently due to
Happel \cite[Theorem 4.6]{Hap91}. We mention that the ``if" part of the second statement follows from \cite[Theorem 6.9(8)]{Bel2000}.

For a self-injective algebra $A$, we have $A\mbox{-Gproj}=A\mbox{-mod}$. Then the final statement, that is also due to \cite[Theorem 2.1]{Ric}, follows.
\end{proof}

Recall that a nontrivial triangulated category $\mathcal{T}$ is \emph{minimal} provided that it has no nontrivial thick subcategories. Lemma \ref{lem:BuchweitzHappel} implies that $A\mbox{-\underline{\rm Gproj}}$ is a thick subcategory of $\mathbf{D}_{\rm sg}(A)$. Then we have the following consequence.

\begin{cor}\label{cor:CM-free}
Let $A$ be a non-Gorenstein algebra such that $\mathbf{D}_{\rm sg}(A)$ is minimal. Then $A$ is CM-free. \hfill $\square$
\end{cor}

Recall from \cite{Ch11'} that a \emph{singular equivalence} between two algebras $A$ and $A'$ means a triangle equivalence between $\mathbf{D}_{\rm sg}(A)$ and $\mathbf{D}_{\rm sg}(A')$.  We observe that a left retraction induces a singular equivalence. The equivalence might be viewed as an enhancement of the isomorphism in Proposition \ref{prop:simple}(5). Here, we recall that for an artin  algebra $A$, the Grothendieck group $K_0(\mathbf{D}_{\rm sg}(A))$ of its singularity category $\mathbf{D}_{\rm sg}(A)$ is isomorphic to ${\rm Cok}\; C_A$; see \cite[4.1]{Hap91}. In view of the remarks after Lemma \ref{lem:quotient}, the following result might be deduced from \cite[Theorem 2.1]{Ch09}.

\begin{prop}\label{prop:sing}
Let $\eta_A\colon A\rightarrow L(A)$ be a left retraction associated to a localizable $A$-module $S$. Identify $S^\perp$ with $L(A)\mbox{-{\rm mod}}$. Then the functors $i^*_\lambda $ and $i^*$ induce mutually inverse triangle equivalences between $\mathbf{D}_{\rm sg}(A)$ and $\mathbf{D}_{\rm sg}(L(A))$.
\end{prop}

\begin{proof}
Consider the triangle functor $i^*_\lambda\colon \mathbf{D}^b(A\mbox{-{\rm mod}})\rightarrow \mathbf{D}^b(L(A)\mbox{-{\rm mod}})$. Denote by ${\rm thick}\langle S\rangle$ the smallest thick subcategory of $\mathbf{D}^b(A\mbox{-{\rm mod}})$ containing $S$. Observe that ${\rm thick}\langle S\rangle\subseteq {\rm perf}(A)$, since ${\rm proj.dim}_A\; S\leq 1$. By Lemma \ref{lem:homo}(1), we infer that $i^*_\lambda({\rm perf}(A))={\rm perf}(L(A))$.

By the recollement in Lemma \ref{lem:recolle}, we infer that $i^*_\lambda$ induces a triangle equivalence $\mathbf{D}^b(A\mbox{-{\rm mod}})/{{\rm thick}\langle S\rangle}\simeq \mathbf{D}^b(L(A)\mbox{-{\rm mod}})$, which restricts to an equivalence ${\rm perf}(A)/{{\rm thick}\langle S\rangle}\simeq {\rm perf}(L(A))$. Then we have done with a canonical triangle equivalence in \cite[Chaptre I, \S 2, 4-3 Corollaire]{Ver77}.
\end{proof}

\section{Nakayama algebras}

 In this section, we recall from \cite{Gus} some homological properties of Nakayama algebras, and introduce the notion of
 $\theta$-perfect element that relates to Gorenstein projective modules. We apply results in the previous section to prove that
 for a  connected Nakayama algebra $A$, there is a connected self-injective Nakayama  algebra $A'$ such that there is a sequence of left retractions linking $A$ to $A'$. Consequently,  the singularity category of $A$ is triangle equivalent to the stable category of $A'$.  We classify Nakayama algebras with at most three simple modules according to the trichotomy: Gorenstein, non-Gorenstein CM-free, non-Gorenstein but not CM-free. It turns out that there is a class of such algebras, that are non-Gorenstein but not CM-free.

\subsection{Nakayama algebras and homological properties} Let $A$ be an artin algebra. An $A$-module is uniserial provided that it  has a unique composition series. Recall that $A$ is \emph{Nakayama} provided that all indecomposable projective and  all indecomposable injective $A$-modules are uniserial, or equivalently, all indecomposable $A$-modules are uniserial.

Assume that $A$ is connected, that is, it does not admit a decomposition as a direct sum of two proper ideals. Then $A$ is Nakayama if and only if its valued quiver is a linear quiver with trivial valuation
\[\xymatrix{
S_1\ar[r] & S_2\ar[r] & \cdots \ar[r] & S_n }\]
or an oriented cycle with trivial valuation
\[\xymatrix{
S_1\ar[r] & S_2\ar[r] & \cdots \ar[r] & S_n. \ar@/^0.7pc/[lll]  }\]
Here, $\{S_1, S_2, \cdots, S_n\}$ is a complete set of representatives of pairwise non-isomorphic simple $A$-modules, and $n=n(A)$ is the number of pairwise non-isomorphic simple $A$-modules. In the first case, $A$ is called a \emph{line algebra}; in the second case, $A$ is a \emph{cycle algebra}; compare \cite{Zac, Mad}. Observe that a line algebra $A$ has finite global dimension; indeed, ${\rm gl.dim}\; A\leq n(A)$.

By an \emph{admissible sequence} of length $n$ we mean a sequence $\mathbf{c}=(c_1, c_2, \cdots, c_n)$ of $n$ positive integers subject to the
conditions $2\leq c_j\leq c_{j+1}+1$ for $j=1,2, \cdots, n-1$ and $c_n\leq c_1+1$. If $c_n\geq 2$, all cyclic permutations of $\mathbf{c}$ are admissible. An admissible sequence $\mathbf{c}$ is \emph{normalized} provided that $c_n=1$,  or $c_1=c_2=\cdots =c_n$, or $c_1$ is minimal among all $c_j$'s and $c_n=c_1+1$.

For a connected Nakayama algebra $A$, we order its simple modules as above. Denote by $P_j$ the projective cover of $S_j$. Then we have exact sequences $P_{j+1}\rightarrow P_j\rightarrow S_j\rightarrow 0$ for $j=1, 2, \cdots, n-1$. For a line algebra, we have  $P_n\simeq S_n$, while for a cycle algebra we have an exact sequence $P_1\rightarrow P_n\rightarrow S_n\rightarrow 0$. Hence, the Nakayama algebra $A$ corresponds to an admissible sequence $\mathbf{c}(A)=(l(P_1), l(P_2), \cdots, l(P_n))$. Moreover, by cyclic permutations, we may always get a normalized admissible sequence. Recall that $A$ is self-injective if and only if $l(P_1)=l(P_2)=\cdots=l(P_n)$.

An indecomposable $A$-module $X$ is uniquely determined by its top ${\rm top}(X)$ and its length $l=l(X)$. Here, the top ${\rm top}(X)=X/{{\rm rad}\; X}=S$ is simple, and  this unique module  $X$ is denoted by $S^{[l]}$.  Then a complete set of representatives of pairwise non-isomorphic indecomposable $A$-modules is given by $\{S_j^{[l]}\; |\; 1\leq j\leq n,\;  l\leq c_j\}$; moreover, the Auslander-Reiten quiver of $A$ is described in \cite[VI. 2]{ARS}.

 In what follows, we recall the minimal projective resolution of an indecomposable module. Recall that $\mathbf{c}(A)=(c_1, c_2, \cdots, c_n)$ is  the admissible sequence of $A$. Following \cite{Gus}, we introduce a map $\theta \colon \{1, 2, \cdots, n\}\rightarrow \{1, 2, \cdots, n\}$ such that $\theta(j)=\phi_n(j+c_j)$. Here, for each positive integer $x$, $\phi_n(x)$ is uniquely determined by the conditions that $1\leq \phi_n(x)\leq n$ and $n$ divides $x-\phi_n(x)$; compare \cite[Section 3]{Mad}. Then we have a descending chain $\{1, 2, \cdots, n\}={\rm Im}\;\theta^0\supseteq {\rm Im}\; \theta \supseteq {\rm Im}\; \theta^2\supseteq {\rm Im}\; \theta^3\supseteq \cdots$. There is a minimal integer $d(A)$ such that ${\rm Im}\; \theta^{d(A)}={\rm Im}\; \theta^{d(A)+1}$. Elements in  ${\rm Im}\; \theta^{d(A)}$ are called $\theta$-\emph{regular elements}. Observe that $0\leq d(A)\leq n(A)-1$, and that $d(A)=0$ if and only if $A$ is self-injective.

For an $A$-module $X$, denote by ${\rm soc}(X)$ the socle of $X$.

\begin{lem} {\rm (\cite{Gus})}\label{lem:top}
Keep the notation as above. Then we have the following statements:
 \begin{enumerate}
 \item ${\rm soc}(P_j)=S_{\theta(j)-1}$ for each $j$, where we identify  $0$ with $n$;
 \item for a nonzero homomorphism $f\colon P_j\rightarrow P_k$, we have that $\theta(j)=\theta(k)$ if $f$ is mono, ${\rm top}({\rm Ker}\; f)=S_{\theta(k)}$ otherwise.
 \end{enumerate}
\end{lem}

\begin{proof}
We use the fact that each indecomposable projective module $P_j$ is uniserial, and that from the top, the $m$-th composition factor in its composition series is $S_{\phi_n(j+m-1)}$.
\end{proof}

We have the following immediate consequence on minimal projective resolutions of indecomposable modules over a Nakayama algebra.

\begin{cor}{\rm (\cite[(5)]{Gus})} \label{cor:resolution}
Let $X$ be an indecomposable non-projective $A$-module. Assume that $X=S_j^{[l]}$ and that $k=\phi_n(j+l)$. Then we have the following statements:
\begin{enumerate}
\item if ${\rm proj.dim}_A\; X=2m$ for $m\geq 1$, then there is an exact sequence
$$0\rightarrow P_{\theta^{m-1}(k)}\rightarrow P_{\theta^{m-1}(j)}  \rightarrow \cdots  \rightarrow P_{\theta(k)} \rightarrow P_{\theta(j)}\rightarrow P_k\rightarrow P_j\rightarrow X\rightarrow 0;$$
\item if ${\rm proj.dim}_A \; X=2m+1$ for $m\geq 0$, then there is an exact sequence
$$0\rightarrow P_{\theta^m(j)}\rightarrow P_{\theta^{m-1}(k)}\rightarrow P_{\theta^{m-1}(j)}  \rightarrow \cdots  \rightarrow P_{\theta(k)} \rightarrow P_{\theta(j)}\rightarrow P_k\rightarrow P_j\rightarrow X\rightarrow 0; $$
\item if ${\rm proj.dim}_A \; X=\infty$, then there is an exact sequence
$$\cdots \rightarrow P_{\theta^m(k)}\rightarrow P_{\theta^{m}(j)}\rightarrow  \cdots  \rightarrow P_{\theta(k)} \rightarrow P_{\theta(j)}\rightarrow P_k\rightarrow P_j\rightarrow X\rightarrow 0.$$
\end{enumerate}
\end{cor}

\begin{proof}
Take an epimorphism  $P_j\stackrel{f}\rightarrow X$. We observe that ${\rm top}({\rm Ker}\; f)=S_k$, and thus we have an exact sequence $P_k\rightarrow P_j\ \rightarrow X\rightarrow 0$. Then we infer the  existence of these sequences by applying   Lemma \ref{lem:top}(2) repeatedly.
\end{proof}

The following result is essentially contained in the proof of \cite[Theorem (i)]{Gus}.

\begin{cor}\label{cor:findim}
Let $A$ be a connected Nakayama algebra with $d(A)$ defined as above. Then we have ${\rm fin.dim}\; A\leq 2d(A)\leq 2n(A)-2$.
\end{cor}

\begin{proof}
We assume the converse, and take $X$ to be an indecomposable $A$-module with  $2d(A)<{\rm proj.dim}_A \; X<\infty$. We use the same notation as in Corollary \ref{cor:resolution}. If ${\rm proj.dim}_A \; X=2m$ with $m\geq d(A)+1$, then both $\theta^{m-1}(k)$ and $\theta^{m-1}(j)$ are $\theta$-regular. By Lemma \ref{lem:top}(2) we have $\theta^{m}(k)=\theta^{m}(j)$. This implies that $\theta^{m-1}(k)=\theta^{m-1}(j)$, since $\theta$ induces a bijection on the set of $\theta$-regular elements. This equality is absurd, since we have a proper monomorphism from $P_{\theta^{m-1}(k)}$ to $P_{\theta^{m-1}(j)}$. A similar argument works for the case ${\rm proj.dim}_A \; X=2m+1$.
\end{proof}

For a nonzero morphism $f\colon P_j\rightarrow P_k$ between two indecomposable projective $A$-modules, we defines its \emph{valuation} $\nu(f)$ as follows.  Take two idempotents $e_j$ and $e_k$ in $A$ such that $P_j\simeq Ae_j$ and $P_k\simeq Ae_k$. Then we have a natural isomorphism ${\rm Hom}_A(P_j, P_k)\simeq e_jAe_k$, which sends $f$ to $f(e_j)$. There is a unique integer $p$ such that $f(e_j)\in {\rm rad}^p\; A$ and $f(e_j) \notin {\rm rad}^{p+1}\; A$. We define $p=\nu(f)$.

\begin{lem}\label{lem:valuation}
Let $f\colon P_j\rightarrow P_k$ be a nonzero morphism as above and let $P_{k'}$ be an indecomposable projective $A$-module. Then we have the following statements:
\begin{enumerate}
\item $\nu(f)=\nu(f^*)$, where $(-)^*={\rm Hom}_A(-, A)$;
\item $l({\rm Cok}\; f)=\nu(f)$, $l({\rm Im}\; f)=l(P_k)-\nu(f)$ and $l({\rm Ker}\; f)=l(P_j)-l(P_k)+\nu(f)$;
\item for any nonzero homomorphism $g\colon P_k\rightarrow P_{k'}$, we have that $\xi\colon P_j\stackrel{f}\rightarrow P_k \stackrel{g}\rightarrow P_{k'}$ is exact if and only if $l(P_{k'})=\nu(f)+\nu(g)$;
\item if both the above sequence $\xi$ and its dual $\xi^*$ are exact, then $l(P_{k'})=l(P_j^*)$.
\end{enumerate}
\end{lem}

\begin{proof}
Recall that $P_j^*\simeq e_jA$ and $P_k^*\simeq e_k A$. Hence, we have the isomorphism ${\rm Hom}_{A^{\rm op}}(P_k^*, P_j^*)\simeq e_jAe_k$ which sends $f^*$ to $f(e_j)$. Then (1) follows.

For (2), we observe that ${\rm Im}\; f={\rm rad}^{\nu(f)} \; P_k$. It follows that $l({\rm Cok}\; f)=\nu(f)$, and then we have the remaining equalities.

For (3), we note that the sequence is exact if and only if $l({\rm Im}\; f)=l({\rm Ker}\; g)$. Applying (2) to $f$ and $g$, we have the result.

By (3),  the exactness of $\xi^*$ implies that $l(P_j^*)=\nu(g^*)+\nu(f^*)$. Using $\nu(g^*)=\nu(g)$ and $\nu(f^*)=\nu(f)$, we are done.
\end{proof}

The following notion is related to Gorenstein projective modules over a Nakayama algebra. A $\theta$-regular element $j$ is called \emph{$\theta$-perfect} provided that $l(P_{\theta^m(j)})=l(P^*_{\theta^{m+1}(j)})$ for all integers $m$. Here, we recall that $\theta$ induces a bijection on the set of $\theta$-regular elements, on which $\theta^{-1}$ is well defined. For a $\theta$-regular element $j$, $\theta^m(j)$ is $\theta$-perfect for any integer $m$.

\begin{prop}\label{prop:comreso}
Let $X$ be an indecomposable Gorenstein projective $A$-module which is non-projective. Assume that $X=S_j^{[l]}$ and that $k=\phi_n(j+l)$. Then the following statements hold:
 \begin{enumerate}
 \item both $j$ and $k$ are $\theta$-perfect;
 \item there is  a complete resolution of $X$ as follows
$$\cdots \rightarrow P_{\theta(k)}\rightarrow P_{\theta(j)}\rightarrow P_k\rightarrow P_j\rightarrow P_{\theta^{-1}(k)} \rightarrow P_{\theta^{-1}(j)} \rightarrow \cdots.$$
\end{enumerate}
\end{prop}

\begin{proof} We have the projective resolution of $X$ as in Corollary \ref{cor:resolution}(3).
By Lemma \ref{lem:Gproj} there is an exact sequence $0\rightarrow X\rightarrow P^1\rightarrow X^1\rightarrow 0$ such that $P^1$ is projective and that $X^1$ is indecomposable Gorenstein projective that is non-projective; moreover, the morphism $P^1\rightarrow X^1$ is a projective cover, and this implies that  $P^1$ is indecomposable. Repeating this argument, we obtain a long exact sequence $0\rightarrow X\rightarrow P^1 \stackrel{d^1}\rightarrow P^2 \stackrel{d^2}\rightarrow P^3\rightarrow \cdots$ such that each $P^j$ is indecomposable projective and each  $X^j={\rm Im}\; d^j$ is Gorenstein projective. Then we have the complete resolution of $X$
$$\cdots \rightarrow P_{\theta(k)}\rightarrow P_{\theta(j)}\rightarrow P_k\rightarrow P_j\rightarrow P^1\rightarrow P^2 \rightarrow \cdots. $$
Here, we recall that an acyclic complex of projective modules is totally acyclic if and only if all its cocycles are Gorenstein projective.

In the above complete resolution, we have the minimal projective resolution of each $X^j$. In view of Corollary \ref{cor:resolution}(3) for $X^j$'s, we have that both $j$ and $k$ are $\theta$-regular, moreover,
$P^{2m-1}=P_{\theta^{-m}(k)}$ and $P^{2m}=P_{\theta^{-m}(j)}$ for $m\geq 1$, and then we have (2). The dual complex of the above resolution is acyclic. Then (1) follows from Lemma \ref{lem:valuation}(4).
\end{proof}

We observe the following immediate consequence of Proposition \ref{prop:comreso}.

\begin{cor}\label{cor:CM-freeNak}
Let $A$ be a connected Nakayama algebra without $\theta$-perfect elements. Then $A$ is CM-free.
\hfill $\square$
\end{cor}

\subsection{Left retractions sequence} Let $A$ be a connected Nakayama algebra, which is not self-injective. Let $\mathbf{c}(A)=(c_1, c_2, \cdots, c_n)$ be its normalized admissible sequence. If $c_n=1$, the simple $A$-module $S_n$ is  projective. Otherwise, $c_n=c_1+1$ and then $S_n$ is localizable with projective dimension one, since we have an exact sequence $0\rightarrow P_1\rightarrow P_n\rightarrow S_n\rightarrow 0$. In both cases, we have that $S_n$ is localizable. Consider the left retraction $\eta_A\colon A\rightarrow L(A)$ associated to $S_n$. We have the following result, a part of which is similar to \cite[Lemmas 7 and 8]{Nag}; compare \cite[Section 2]{BFVZ}.

 For an admissible sequence $\mathbf{c}$ of length $n$, set $\mathbf{c}'=(c'_1, c'_2, \cdots, c'_{n-1})$ such that $c'_j=c_j-[\frac{c_j+j-1}{n}]$. Here, for a real number $x$, $[x]$ denotes the largest integer that is not strictly larger than $x$. Then $\mathbf{c}'$ is an admissible sequence of length $n-1$.

\begin{lem}\label{lem:Nakayama}
Keep the notation as above. Then we have the following statements:
\begin{enumerate}
\item the functor $i_\lambda\colon A\mbox{-}{\rm mod}\rightarrow L(A)\mbox{-}{\rm mod}$ sends indecomposable modules, that are not isomorphic to $S_n$, to indecomposable modules; more precisely, we have that $i_\lambda(S_j^{[l]})=(i_\lambda(S_j))^{[l']}$ with $1 \leq j\leq n-1$ and $l'=l-[\frac{l+j-1}{n}]$, and that $i_\lambda(S_n^{[l]})=(i_\lambda(S_1))^{[l']}$ with $l'=l-1-[\frac{l-1}{n}]$;
\item  the algebra $L(A)$ is connected Nakayama with  $\mathbf{c}(L(A))=\mathbf{c}(A)'$;
\item we have $d(L(A))\leq d(A)\leq d(L(A))+1$; here, $d(A)$ is defined as in Subsection 3.1.
\item  $L(A)$ is a line algebra if and only if $c_{n-1}=2$, that is equivalent to
the fact that $A$ is a line algebra or $\mathbf{c}(A)=(2, 2, \cdots, 2, 3)$.
\end{enumerate}\end{lem}

\begin{proof}
 Set $S=S_n$. We identify $S^\perp$ with $L(A)\mbox{-mod}$, and the functor $i_\lambda$ with the quotient functor $q\colon A\mbox{-mod}\rightarrow A\mbox{-mod}/{{\rm add}\; S}$; see Lemma \ref{lem:quotient}.

 For (1), it suffices to recall a general fact: given  any abelian category $\mathcal{A}$ and a Serre subcategory $\mathcal{C}$, the corresponding quotient functor $q\colon \mathcal{A}\rightarrow \mathcal{A}/\mathcal{C}$ sends a uniserial object to a uniserial object. For any $A$-module $X$, we observe that the length of $i_\lambda(X)$ equals the length of $X$ minus the multiplicity of $S$ in a composition series of $X$.

 Any indecomposable $L(A)$-module is isomorphic to the image under $i_\lambda$ of some indecomposable $A$-module. Then it follows from (1) that the algebra $L(A)$ is Nakayama. By Lemma \ref{lem:quiver}, the valued quiver of $L(A)$ is connected and then $L(A)$ is connected. The statement about  admissible sequences in (2) follows from (1).

 For (3), set $d=d(A)$ and $d'=d(L(A))$. Consider the map $\theta'\colon \{1, 2, \cdots, n-1\}\rightarrow \{1, 2, \cdots, n-1\}$ given by $\theta'(j)=\phi_{n-1}(j+c'_j)$. Define a surjective map $\pi\colon \{1, 2, \cdots, n\}\rightarrow \{1, 2, \cdots, n-1\}$ such that $\pi(n)=1$ and $\pi(j)=j$ for $j< n$. We claim that $\pi\circ \theta=\theta'\circ \pi$. This follows from (2) and the following fact: for an integer $m$, we have $\phi_n(m)=\phi_{n-1}(m-[\frac{m-1}{n}])$, if $\phi_n(m)<n$; otherwise, we have $\phi_{n-1}(m-[\frac{m-1}{n}])=1$.

We infer from the claim that ${\rm Im}\; \theta'^m=\pi({\rm Im}\; \theta^m)$ for any $m\geq 0$. By the definition of $d=d(A)$, we have ${\rm Im}\; \theta^d={\rm Im}\; \theta^{d+1}$, and thus ${\rm Im}\; {\theta'}^d={\rm Im}\; {\theta'}^{d+1}$. It follows that $d'\leq d$. By ${\rm Im}\; {\theta'}^{d'}={\rm Im}\; {\theta'}^{d'+1}$, we infer that ${\rm Im}\; \theta^{d'}$ and ${\rm Im}\; \theta^{d'+1}$ have the same image under $\pi$. Then there are three possibilities: ${\rm Im}\; \theta^{d'}={\rm Im}\; \theta^{d'+1}$, ${\rm Im}\; \theta^{d'}={\rm Im}\; \theta^{d'+1} \cup \{ n \}$ or ${\rm Im}\; \theta^{d'}={\rm Im}\; \theta^{d'+1} \cup \{ 1 \}$. In each case, we have that  ${\rm Im}\; \theta^{d'+1}={\rm Im}\; \theta^{d'+2}$. Hence, $d\leq d'+1$. This proves (3).

For (4), we observe that $L(A)$ is a line algebra if and only if $c'_{n-1}=1$, that is equivalent to $c_{n-1}=2$. Recall by assumption that $\mathbf{c}(A)$ is normalized. Then we are done.
\end{proof}

The following result associates to any connected Nakayama algebra a self-injective one, via a sequence of left retractions.

\begin{thm}\label{thm:2}
Let $A$ be a connected Nakayama algebra with $n(A)$ the number of pairwise non-isomorphic simple $A$-modules and $C_A$ the Cartan matrix. Recall the notation $d(A)$ in Subsection 3.1. Then there is a sequence of homomorphisms between connected Nakayama algebras
\begin{align}\label{equ:lr}
A=A_0\stackrel{\eta_0}\longrightarrow A_1 \stackrel{\eta_1}\longrightarrow  A_2\longrightarrow \cdots \stackrel{\eta_{r-1}}\longrightarrow A_r
\end{align}
such that $d(A)\leq r\leq n(A)-1$,  each $\eta_i$ is a left retraction associated to some localizable module, and that $A_r$ is self-injective.  Moreover, we have
\begin{enumerate}
\item the algebra $A_r$ is simple if and only if ${\rm gl.dim}\; A< \infty$;  in this case, $r=n(A)-1$;
\item if ${\rm gl.dim}\; A=\infty$, the composition $\eta_{r-1}\circ \cdots \circ \eta_1\circ \eta_0\colon A\rightarrow A_r$ is uniquely determined by $A$, and  $r=\sharp \{1\leq j \leq n(A)\; |\; {\rm proj.dim}_A\; S_j<\infty\}$.
\end{enumerate}
\end{thm}

We will denote by $r(A)$ the unique  $r$ of the above sequence (\ref{equ:lr}). Hence, we have $d(A)\leq r(A)\leq n(A)-1$. For $A$ with infinite global dimension, $r(A)$ equals the number of $\Psi$-regular simple modules; see \cite[Corollary 3.6]{Mad}.

\begin{proof}
Denote by  $\{S_1, S_2, \cdots, S_n\}$ a complete set of representatives of pairwise non-isomorphic simple $A$-modules. If $A$ is self-injective, we set $r=0$ and $A=A_r$. If $A$ is not self-injective, we apply Lemma \ref{lem:Nakayama} repeatedly
to obtain such a sequence. We apply Proposition \ref{prop:simple}(1) repeatedly to get $n(A_r)=n(A)-r\geq 1$, and then $r\leq n(A)-1$. Since $A_r$ is self-injective, we have $d(A_r)=0$. Then applying Lemma \ref{lem:Nakayama}(3) we have $d(A)\leq r$.

Applying Lemma \ref{lem:homo}(2) repeatedly, we have that ${\rm gl.dim}\; A< \infty$ if and only if   ${\rm gl.dim}\; A_r< \infty$; this is equivalent to that $A_r$ is simple, since it is connected self-injective. In this case, $n(A_r)=1$, and thus $r=n(A)-1$.

Assume that ${\rm gl.dim}\; A=\infty$. Then each simple $A_r$-module has infinite projective dimension. Consider the adjoint pair $(F, G)$, where $F= A_r\otimes_A-\colon A\mbox{-mod}\rightarrow A_r\mbox{-mod}$ and $G={\rm Hom}_{A_r}(A_r, -)\colon A_r\mbox{-mod}\rightarrow A\mbox{-mod}$. By Lemma \ref{lem:quotient} and the remarks afterward, the right $A$-module $A_r$ is projective and then $F$ is exact; moreover, $F$ identifies with the quotient functor $q\colon A\mbox{-mod}\rightarrow A\mbox{-mod}/{{\rm Ker}\; F}$, and the functor $G$ is fully faithful. Here, ${\rm Ker}\; F$ is the Serre subcategory formed by $A$-modules on which $F$ vanishes. It follows from \cite[Proposition 2.2]{GL1991} that the image of the functor $G$ is $({\rm Ker}\; F)^\perp=\{X\in A\mbox{-mod}\; |\; {\rm Hom}_A(M, X)=0={\rm Ext}_A^1(M, X) \mbox{ for all } M\in {\rm Ker}\; F\}$.

We claim that ${\rm Ker}\; F=\langle S_j\; |\; {\rm proj.dim}_A\; S_j< \infty, j=1, 2, \cdots, n(A)\rangle$.  Here, for a class $\mathcal{S}$ of $A$-modules, we denote by $\langle \mathcal{S}\rangle$ the smallest Serre subcategory containing $\mathcal{S}$. Observe that ${\rm Ker}\; F=\langle S_j\; |\; S_j\in {\rm Ker}\; F\rangle$. We identify $F$ with the quotient functor $q$. Then for any simple $A$-module $S_j$, $F(S_j)$ is either zero or a simple $A_r$-module. Applying Lemma \ref{lem:homo}(1) repeatedly, ${\rm proj.dim}_A\; S_j<\infty$ if and only if ${\rm proj.dim}_{A_r}\; F(S_j)<\infty$, which is equivalent to $F(S_j)=0$. Then we are done with the claim.  We observe from the proof  that $n(A_r)=n(A)-\sharp \{1\leq j \leq n(A)\; |\; {\rm proj.dim}_A\; S_j<\infty\}$, and thus $r=\sharp \{1\leq j \leq n(A)\; |\; {\rm proj.dim}_A\; S_j<\infty\}$; compare Proposition \ref{prop:simple}(1).

We conclude that the image of the fully faithful functor $G={\rm Hom}_A(A_r, -)$ is uniquely determined by $A$. Then the uniqueness of the composite homomorphism $A\rightarrow A_r$ follows from Lemma \ref{lem:Gabriel-Pena}.
\end{proof}

\begin{lem}\label{lem:Gabriel-Pena}
Let $\phi\colon A\rightarrow B$ and $\phi'\colon A\rightarrow B'$ be two algebra homomorphisms of artin algebras. Assume that both functors ${\rm Hom}_A(B, -)$ and ${\rm Hom}_A(B', -)$ are fully faithful with the same image in $A\mbox{-}{\rm mod}$. Then there is an isomorphism $\psi\colon B\rightarrow B'$ of algebras such that $\psi\circ \phi=\phi'$
\end{lem}

\begin{proof}
This is analogous to the bijection in \cite[Theorem 1.2]{GP}.
\end{proof}

We draw some consequences of Theorem \ref{thm:2}. In the following result, the first statement is contained in \cite[Theorem 6]{BFVZ}, and the second is due to \cite{Gus}. By ${\rm max}\; \mathbf{c}(A)$ we mean the maximum of $c_j=l(P_j)$ for $1\leq j\leq n(A)$; it is the Lowey length of the algebra $A$.

\begin{cor}
Let $A$ be a connected Nakayama algebra. Then we have the following statements:
\begin{enumerate}
\item ${\rm gl.dim}\; A< \infty$ if and only if ${\rm det}\; C_A=1$;
\item if ${\rm gl.dim}\;A < \infty$, then  ${\rm gl.dim}\; A\leq 2n(A)-2$ and ${\rm max}\; \mathbf{c}(A)\leq 2n(A)-1$.
\end{enumerate}\end{cor}

\begin{proof}
Consider the admissible sequence $\mathbf{c}(A_r)=(c, c, \cdots, c)$ with $n(A_r)$ copies of $c$. Then in the Cartan matrix $C_{A_r}$, the sum of all entries in each column equals $c$. This implies that $c$ divides ${\rm det }\; C_{A_r}$. It follows that ${\rm det }\; C_{A_r}=1$ if and only if $\mathbf{c}(A_r)=(1)$, that is,  $A_r$ is simple. By Proposition \ref{prop:simple}(5), we have ${\rm det }\; C_{A}={\rm det }\; C_{A_r}$. Then (1) follows from Theorem \ref{thm:2}(1).

The first half in (2) follows from Corollary \ref{cor:findim}. For the second half, we may assume that $A$ is a cycle algebra with its normalized admissible sequence $\mathbf{c}(A)$.  Using induction, we assume that ${\rm max}\; \mathbf{c}(A_1)\leq 2n(A)-3$. Then  by Lemma \ref{lem:Nakayama}(2), we have  that each  $c_1< 2n(A)-1$ and $c_j\leq  2n(A)-1$ for $2\leq j\leq n(A)-1$. Since $c_n=c_1+1$, we have $c_n\leq 2n(A)-1$.
\end{proof}

Recall that $A_r$ is self-injective. Then the stable module category $A_r\mbox{-\underline{mod}}$ has a canonical triangulated structure.

\begin{cor}\label{cor:NakaSing}
Let $A$ be a connected Nakayama algebra. Then the composite homomorphism in (\ref{equ:lr}) induces  a triangle equivalence
$$\mathbf{D}_{\rm sg}(A)\simeq A_r\mbox{-\underline{\rm mod}}.$$
\end{cor}

It follows that the singularity category $\mathbf{D}_{\rm sg}(A)$ is Krull-Schmidt; its Auslander-Reiten quiver is a truncated tube of rank $n(A)-r(A)$. Moreover, it is a homogeneous tube, that is, a tube of rank one, if and only if $n(A_r)=1$ and $A_r$ is not simple.

\begin{proof}
Recall from Lemma \ref{lem:BuchweitzHappel}  the triangle equivalence $\mathbf{D}_{\rm sg}(A_r)\simeq A_r\mbox{-\underline{\rm mod}}$. Then the above  triangle equivalence follows from Proposition  \ref{prop:sing}.
\end{proof}

\subsection{Nakayama algebras with at most three simples} Let $A$ be a connected Nakayama algebra with two simple modules, that is, $n(A)=2$. We may assume that $A$ is a cycle algebra which is not self-injective. This implies that its normalized admissible sequence  is given by $\mathbf{c}(A)=(c, c+1)$ for $c\geq 2$.

\begin{lem}\label{lem:n2}
Keep the assumption as above. Then there are two cases:
\begin{enumerate}
\item if $\mathbf{c}(A)=(2k, 2k+1)$ for $k\geq 1$, then $A$ is Gorenstein with ${\rm v.dim}\; A=2$;  $A$ has finite global dimension if and only if $k=1$;
\item  if $\mathbf{c}(A)=(2k+1, 2k+2)$ for $k\geq 1$, then $A$ is non-Gorenstein CM-free  with  ${\rm fin.dim}\; A=1$.
\end{enumerate}
\end{lem}

\begin{proof}
We consider the retraction $\eta\colon A\rightarrow L(A)$ associated to $S_2$. Observe that $L(A)$ has a unique simple module with $\mathbf{c}(L(A))=([\frac{c}{2}])$, and that $L(A)$ is self-injective; moreover, the stable category $L(A)\mbox{-\underline{mod}}$, as a triangulated category, is minimal. By Corollary \ref{cor:NakaSing}, the singularity category $\mathbf{D}_{\rm sg}(A)$ is minimal; here $A_r=L(A)$. Then $A$ has finite global dimension if and only if $L(A)$ is simple, that is, $[\frac{c}{2}]=1$. This is equivalent to that $\mathbf{c}(A)=(2,3)$.

For (1), we have that the length of the indecomposable $L(A)$-module $i_\lambda(I_2)$ equals $k=[\frac{c}{2}]$, and it is projective. By Theorem \ref{thm:1} $A$ is Gorenstein. The same reasoning yields that the algebra $A$ in (2) is non-Gorenstein. In this case, by the minimality of  $\mathbf{D}_{\rm sg}(A)$, we infer that $A$ is CM-free; see Corollary \ref{cor:CM-free}.
\end{proof}

In what follows, we assume that $A$ is a cycle algebra with $n(A)=3$ which is not self-injective. Then we may assume that the normalized admissible sequence is $\mathbf{c}(A)=(c, c+j, c+1)$ with $c\geq 2$ and $j=0, 1, 2$.

\begin{cor}\label{cor:finitegldim}
Keep the assumption as above. Then $A$ has finite global dimension if and only if its normalized admissible sequence $\mathbf{c}(A)$ is $(2,2, 3)$, $(2,4,3)$, $(3, 4,4)$ or  $(3, 5,4)$. In this case, the global dimension of $A$ equals $3$, $2$, $4$ and $2$, respectively.
\end{cor}

\begin{proof}
 Recall that $A$ has finite global dimension if and only if so does $L(A)$. This is equivalent to that  $\mathbf{c}(L(A))$ is $(2,1)$ or $(2,3)$ up to cyclic permutations; see Lemma \ref{lem:n2}(1). Applying $\mathbf{c}(L(A))=\mathbf{c}(A)'$, we infer the result immediately.
\end{proof}

The following result classifies connected Nakayama algebras with three simple modules according to the trichotomy: Gorenstein, non-Gorenstein CM-free, non-Gorenstein but not CM-free.

\begin{prop}\label{prop:classification}
Let $A$ be a cycle algebra with $n(A)=3$ which is not self-injective. Denote by $\mathbf{c}(A)$ its normalized admissible sequence. Then we have the following statements:
\begin{enumerate}
\item the algebra $A$ is Gorenstein if and only if $\mathbf{c}(A)=(2,2,3)$, $(2,4,3)$,  $(3k, 3k, 3k+1)$, $(3k, 3k+1, 3k+1)$, $(3k, 3k+2, 3k+1)$ or  $(3k+1, 3k+2, 3k+2)$ for $k\geq 1$;
\item the algebra $A$ is non-Gorenstein CM-free if and only if $\mathbf{c}(A)=(2, 3, 3)$, $(3k+1,3k+1, 3k+2)$, $(3k+1, 3k+3, 3k+2)$,  $(3k+2, 3k+2, 3k+3)$   or  $(3k+2, 3k+4, 3k+3)$ for $k\geq 1$;
\item the algebra $A$ is non-Gorenstein, but not CM-free if and only if $\mathbf{c}(A)=(3k+2, 3k+3, 3k+3)$ for $k\geq 1$; in this case, all indecomposable non-projective Gorenstein projective $A$-modules are given by $S_2^{[3m]}$ for $1\leq m\leq k$.
\end{enumerate}
In case (1), the virtual dimension ${\rm v.dim}\; A$ equals $3$, $2$, $2$, $4$, $2$ and $2$, respectively.
\end{prop}

\begin{proof}
It suffices to prove the ``if" part of all the statements. For (1), by Corollary \ref{cor:finitegldim} if  $\mathbf{c}(A)=(2,2,3)$ or $(2,4,3)$,  the algebra $A$ has finite global dimension, and thus it is  Gorenstein. We consider the case $\mathbf{c}(A)=(3k, 3k, 3k+1)$ for $k\geq 1$. Then the left retraction $L(A)$ of $A$ with respect to $S_3$ satisfies that $\mathbf{c}(L(A))=(2k, 2k)$,  and thus $L(A)$ is self-injective. Observe that $I_3$ is projective. Then the $L(A)$-module $i_\lambda(I_3)$ is projective. By Theorem \ref{thm:1} the algebra $A$ is Gorenstein. Similar argument works for the remaining three cases.

Recall that a CM-free algebra is non-Gorenstein if and only if it has infinite global dimension. For (2), assume that $\mathbf{c}(A)=(3k+1, 3k+3, 3k+2)$ with $k\geq 1$. Then we have $\mathbf{c}(L(A))=(2k+1, 2k+2)$ and by Lemma \ref{lem:n2}, $L(A)$ is CM-free of infinite global dimension. Thus by Propositions \ref{prop:Gproj} and Lemma \ref{lem:homo}(2), $A$ is CM-free of infinite global dimension. Similar argument works for the case $\mathbf{c}(A)= (3k+2, 3k+2, 3k+3)$ or $(3k+2, 3k+4, 3k+3)$.

We consider the case $\mathbf{c}(A)=(3k+1,3k+1, 3k+2)$ for $k\geq 1$. Then we have  $\mathbf{c}(L(A))=(2k+1, 2k+1)$, and thus $L(A)$ is self-injective and has infinite global dimension. So by Lemma \ref{lem:homo}(2), the algebra  $A$ has infinite global dimension. Observe that the set of  $\theta$-regular elements is  $\{2, 3\}$,  and $\theta$ sends $2$ to $3$, and $3$ to $2$; moreover, $l(P_2^*)=3k+1$ and $l(P_3^*)=3k+1$, where $(-)^*={\rm Hom}_A(-, A)$. Then we infer that there are no $\theta$-perfect elements.  By Corollary \ref{cor:CM-freeNak} the algebra  $A$ is CM-free. The only remaining case in (2) is $(2, 3, 3)$, which follows from the following argument for (3) (take $k$ to be zero).

For (3), we assume that $\mathbf{c}(A)=(3k+2, 3k+3, 3k+3)$. Then we have $\mathbf{c}(L(A))=(2k+2, 2k+2)$, and thus $L(A)$ is self-injective. Observe that $I_3=S_2^{[3k+2]}$. It follows that the length of $i_\lambda(I_3)$ is $2k+1$; see Lemma \ref{lem:Nakayama}(1). Hence, the $L(A)$-module $i_\lambda(I_3)$ is not projective and thus has infinite projective dimension. It follows from Theorem \ref{thm:2} that $A$ is non-Gorenstein. Observe that the modules $S_2^{[3m]}$ are Gorenstein projective, whose complete resolution is periodic as follows
$$\cdots \rightarrow P^2 \stackrel{f}\rightarrow P^2\stackrel{g}\rightarrow P^2 \stackrel{f}\rightarrow P^2 \stackrel{g}\rightarrow P^2 \rightarrow \cdots .$$
Here, we have that $\nu(f)=3(k-m+1)$ and $\nu(g)=3m$.

We claim that any indecomposable Gorenstein projective $A$-module $X$, that is not projective, is of the form $S_2^{[3m]}$. Indeed, the set of $\theta$-regular elements are $\{2, 3\}$, on which $\theta$ acts as the identity. Observe that $l(P_2^*)=3k+3$ and $l(P_3^*)=3k+2$. Then the only $\theta$-perfect element is $2$. It follows from Proposition \ref{prop:comreso} that $X$ fits into an exact sequence $P^2\rightarrow P^2\rightarrow X\rightarrow 0$. This implies that $X$ is isomorphic to $S_2^{[3m]}$ for $1\leq m\leq k$. Then we are done.
\end{proof}

\bibliography{}

\vskip 10pt

 {\footnotesize \noindent Xiao-Wu Chen, Yu Ye \\
  School of Mathematical Sciences, University of Science and Technology of
China, Hefei 230026, Anhui, PR China \\
Wu Wen-Tsun Key Laboratory of Mathematics, USTC, Chinese Academy of Sciences, Hefei 230026, Anhui, PR China.\\
URL: http://home.ustc.edu.cn/$^\sim$xwchen, http://staff.ustc.edu.cn/$^\sim$yeyu}

\end{document}